\documentclass{article}
\usepackage{amssymb,amsthm,amsmath,mathtools}
\usepackage{hyperref}
\usepackage{authblk}
\usepackage{bbm}
\usepackage{enumitem}
\usepackage{fullpage}
\usepackage{bbm}
\usepackage{orcidlink}

\usepackage{pgf,tikz}
\usepackage[numeric,initials,nobysame,msc-links,abbrev]{amsrefs}

\newtheorem{theorem}{Theorem}
\newtheorem{corollary}[theorem]{Corollary}
\newtheorem{lemma}[theorem]{Lemma}
\newtheorem{proposition}[theorem]{Proposition}

\theoremstyle{definition}
\newtheorem{definition}[theorem]{Definition}
\newtheorem{remark}[theorem]{Remark}

\newtheorem{claim}[theorem]{Claim}

\newcommand{\cE}{\ensuremath{\mathcal E}}
\newcommand{\cF}{\ensuremath{\mathcal F}}

\newcommand{\cM}{\ensuremath{\mathcal M}}
\newcommand{\cN}{\ensuremath{\mathcal N}}

\newcommand{\cP}{\ensuremath{\mathcal P}}
\newcommand{\cQ}{\ensuremath{\mathcal Q}}

\newcommand{\cV}{\ensuremath{\mathcal V}}

\newcommand{\bbN}{{\ensuremath{\mathbb N}} }

\newcommand{\bbP}{{\ensuremath{\mathbb P}} }

\newcommand{\bbZ}{{\ensuremath{\mathbb Z}} }

\newcommand{\bA}{\ensuremath{\mathbf{A} }}

\newcommand{\E}{\ensuremath{\mathcal{E}}}
\newcommand{\1}{\mathbbm1}

\setlist[itemize]{leftmargin=*}
\setlist[enumerate]{leftmargin=*,label=(\roman*),ref=(\roman*)}

\newcommand{\pc}{p_{\mathrm{c}}}
\newcommand{\md}{\mathrm{d}}

\title{Sharpness of the phase transition for constrained-degree percolation}
\author[1]{Ivailo Hartarsky}
\author[2]{Roger W. C. Silva}
\affil[1]{\small Université Claude Bernard Lyon 1, CNRS, Centrale Lyon, INSA Lyon, Université Jean Monnet, ICJ UMR5208, 69622 Villeurbanne, France
\texttt{hartarsky@math.univ-lyon1.fr} - \orcidlink{0000-0002-7480-1500}} 
\affil[2]{\small Universidade Federal de Minas Gerais, Departamento de Estatística,  Av. Antônio Carlos 6627, Belo Horizonte, Brazil \texttt{rogerwcs@est.ufmg.br} - \orcidlink{0000-0002-6365-9211}}

\date{\today}

\begin{document}

\maketitle

\begin{abstract} We consider constrained-degree percolation on the hypercubic lattice. Initially, all edges are closed, and each edge independently attempts to open at a uniformly distributed random time; the attempt succeeds if, at that instant, both end-vertices have degrees strictly less than a prescribed parameter. The absence of the FKG inequality and the finite energy property, as well as the infinite range of dependence, make the rigorous analysis of the model particularly challenging. In this work, we show that the one-arm probability exhibits exponential decay in its entire subcritical phase. The proof relies on the Duminil-Copin--Raoufi--Tassion randomized algorithm method and resolves a problem of dos Santos and the second author. At the heart of the argument lies an intricate combinatorial transformation of pivotality in the spirit of Aizenman--Grimmett essential enhancements, but with unbounded range. This technique may be of use in other dynamical settings.
\end{abstract}


\section{Introduction}
Percolation models have been extensively studied for decades. Since their introduction in \cite{Broadbent57}, they have become central objects of research in probability and mathematical physics, with numerous variants of the original model proposed over time. Several models (see e.g.\ \cites{Jahnel23,Lichev24,Peled19,Garet18,Pete08,Grimmett10,Grimmett17,Holroyd21,Li20}) purposefully feature degree constraints, often making their analysis intricate. Our model of interest is the \emph{constrained degree percolation (CDP)} introduced in \cite{Teodoro14}. In CDP, each vertex may accept connections up to a fixed limit, after which no new edges can be formed, while all previously established connections remain intact. Specifically, it is a dependent continuous-time percolation process, defined by the following dynamics. At time zero, all edges are closed; each edge attempts to open at a uniformly distributed random time, and the attempt is successful if, at that moment, both of its end-vertices have degrees strictly smaller than a given parameter $\kappa$. A formal definition of the model is provided in Section \ref{model}.
Unlike classical Bernoulli percolation, the CDP exhibits dependencies of all orders at any fixed time. Furthermore, the absence of the Fortuin–Kasteleyn–Ginibre (FKG) inequality and the lack of the finite energy property make its analysis challenging.

\subsection{Model and result}\label{model}

Let us start by introducing the CDP on $\bbZ^d$, $d\geq 2$. Let $2\leq \kappa\leq 2d$ be an integer, which is fixed throughout the paper and kept implicit. We denote by $\cE$ the set of nearest neighbour (non-oriented) edges of $\bbZ^d$.

Let $U=(U_e)_{e\in\E}$ be a sequence of independent uniform random variables on $[0,1]$ with corresponding product measure $\bbP$. We denote by $\Omega$ the set of $U\in[0,1]^\cE$ with distinct coordinates and such that there is no infinite walk\footnote{A \emph{walk} is a sequence of vertices such that consecutive ones are joined by edges. A \emph{path} is a walk consisting of distinct vertices.} $x_1,x_2,\dots$ in $\bbZ^d$ such that $i\mapsto U_{x_ix_{i+1}}$ is strictly monotone. It is classical that $\bbP(\Omega)=1$, so we systematically restrict our attention to $U\in\Omega$. 

For any $p\in[0,1]$, we define the function $\omega_p:\Omega\to\{0,1\}^\cE$ as follows. We write $\omega_{p}(U,e)$ for the value of $\omega_p(U)$ at edge $e\in\cE$ and refer to $e$ as \emph{$p$-open} if $\omega_p(U,e)=1$, and \emph{$p$-closed} if $\omega_p(U,e)=0$. Initially, only the (at most one) edge with $U_e=0$ is open, that is, $\omega_0(U,e)=\1_{U_e=0}$. Each edge $e\in\cE$ attempts to open at time $U_e$ and is successful only if its addition to the currently open edges would not create a vertex of degree larger than $\kappa$. In other words, $p\mapsto \omega_p(U,xy)$ jumps from $0$ to $1$ at $p=U_{xy}$ if 
\[|\{z\in\bbZ^d:xz\in\cE,\omega_{U_{xy}-}(U,xz)=1\}|<\kappa\text{ and }|\{z\in\bbZ^d:yz\in\cE,\omega_{U_{xy}-}(U,yz)=1\}|<\kappa,\]
and stays equal to $0$ for all $p\in[0,1]$ otherwise. This definition is well-posed in $\Omega$, since for every $e\in\cE$ only a finite set of coordinates of $U$ suffices to determine $(\omega_p(U,e))_{p\in[0,1]}$.

We call a path \emph{$p$-open} if all of its edges are $p$-open. The event that $x\in\bbZ^d$ is connected to $y\in\bbZ^d$ by a $p$-open path is denoted by $x\xleftrightarrow{p}y$. We write $x\xleftrightarrow p \infty$ for the event that there exists an infinite $p$-open path starting at $x$. We consider the standard \emph{one-arm} probability
\begin{align}
\label{one_arm}\theta_n(p)&{}=\bbP(\exists x\in\bbZ^d,\|x\|_1=n,0\xleftrightarrow px),&\theta(p)&{}=\bbP(0\xleftrightarrow p\infty)=\lim_{n\to\infty}\theta_n(p).\end{align}
Since $\theta:[0,1]\to[0,1]$ is clearly non-decreasing, we naturally define 
\[\pc=\sup\{p\in[0,1]:\theta(p)=0\},\]
with the convention $\sup\varnothing=0$, which will never enter into effect. Our main result is the following.

\begin{theorem}[Sharp phase transition]
\label{th:main}Consider the CDP on $\bbZ^d$, $d\geq 2$, with $\kappa\in\{2,\dots,2d\}$. For $p<\pc$, there exists $\alpha_p>0$ such that for all $n\geq 1$, $\theta_n(p)\leq e^{-n\alpha_p}$.
\end{theorem}
\begin{remark}
    Theorem~\ref{th:main} is stated and proved for $\bbZ^d$ with its usual graph structure for concreteness. However, the proof applies mutatis mutandis for any transitive locally finite bipartite graph. For non-bipartite graphs, the structure of the switching path from Lemma~\ref{lem:path} below becomes more complex. Nonetheless, we expect that our argument can be adapted to this case as well.
\end{remark}

As it is common in percolation, once sharpness is established, many questions about the subcritical phase become readily accessible. We illustrate this by the following result on the exponential decay of the volume of the cluster of the origin.

\begin{corollary}[Volume exponential decay]
    \label{cor:size}
    Consider the CDP on $\bbZ^d$, $d\ge 2$, $\kappa\in\{2,\dots,2d\}$, $p<\pc$. Set
    \[C_0=\{x\in\bbZ^d:0\xleftrightarrow px\}.\]
    Then there exists $\gamma_p>0$ such that for all $n\ge 1$, 
    \[\bbP(|C_0|\ge n+1)\le e^{-n\gamma_p}.\]
\end{corollary}
Corollary~\ref{cor:size} is proved in Appendix~\ref{app}.

\subsection{Background}
For $\kappa=2d$, the CDP reduces to the standard Bernoulli percolation model, for which Theorem~\ref{th:main} is a classical result \cites{Aizenman87,Duminil-Copin16a,Menshikov86}. We direct the reader to \cite{Grimmett99} for some background on this model.

It is known that the CDP undergoes a nontrivial phase transition for all $d\geq 2$ and most values of $\kappa$. For example, \cite{Delima20} establishes this nontriviality for $d=2$ and $\kappa=3$, while also proving that percolation does not occur when $d\ge\kappa=2$, even at $p=1$. Further progress was made in \cite{Hartarsky22constrainedperco}, which provides quantitative upper bounds on the critical time and characterizes the phase transition for all $d\geq 3$ and most choices of $\kappa\ge3$. Nonetheless, it remains open to determine whether the phase transition is trivial ($\pc=1$) for $\kappa\in\{3,\dots,9\}$ for some dimensions, but it is conjectured that this is never the case. In view of this, let us emphasise that Theorem~\ref{th:main} applies regardless of whether the phase transition is trivial.

It is straightforward that the CDP is stochastically dominated by standard Bernoulli percolation for every fixed $p\in[0,1]$. Consequently, Theorem \ref{th:main} holds trivially whenever $p$ lies below the Bernoulli percolation critical threshold, denoted by $\overline\pc$. However, it is known that $\pc>\overline\pc$ (see Theorem 1 in \cite{Delima20} for the case $d=2$ and $\kappa=3$; the argument extends for the general case), and therefore Theorem~\ref{th:main} remains nontrivial for all choices of $d$ and $\kappa$.

Let $\widehat{\pc}=\sup\{p\in[0,1]: \sum_{x\in\bbZ^d}\bbP(0\xleftrightarrow{p}x)<\infty\}\le\pc$ denote the \emph{susceptibility critical threshold} for the CDP. The recent work \cite{DosSantos25} proves Theorem~\ref{th:main} for $p<\widehat{\pc}$. Our result strengthens their findings: in particular, Theorem~\ref{th:main} implies $\pc=\widehat{\pc}$. This answers \cite{DosSantos25}*{Question 1.} in the setting of homogeneous (vertex-independent) constraints $\kappa$.

\subsection{Outline of the proof}
\label{sec:outline}
A standard approach to proving sharpness in percolation models is the one of \cite{Duminil-Copin19} using a randomised algorithm (see Section~\ref{sec:DCRT}) and the OSSS inequality \cite{ODonnell05} (see Section~\ref{sec:OSSS}). Applying this method directly fails. On the one hand, the measure $\omega_p\circ\bbP$ is not monotonic and lacks other good properties. On the other hand, if we choose to work with the configuration space $[0,1]^\cE$ instead of $\{0,1\}^\cE$, edges revealed in the exploration of the occurrence of the one-arm event in \eqref{one_arm} do not witness this one-arm event on a smaller scale. For this reason, as in \cite{Hartarsky21a}, we introduce an alternative one-arm event (see Section~\ref{sec:one-arm}) for which revealment is naturally expressed in terms of its probability. Namely, this event requires a path whose first part (head) is open and whose second part (tail) has decreasing values of the uniform variables on edges. This modified one-arm event satisfies a Russo formula (see Lemma~\ref{lem:russo}).

Putting the above ingredients together, one reaches the main difficulty: the Russo formula features a different pivotality event ($p$-pivotality) from the influences appearing in the OSSS inequality ($U$-pivotality). In Proposition~\ref{prop:pivotality:transfer}, we provide a transfer from one to the other. The statement is largely inspired by the essential enhancement technique of \cite{Aizenman91}. However, the lack of finite-energy property and the infinite range of dependence make its implementation very delicate (see e.g.\ \cite{Duminil-Copin18} for a very different setting where an intricate unbounded range pivotality transfer is also performed). The base idea of the proof of Proposition~\ref{prop:pivotality:transfer} is simple. The probability that the state of an edge depends on the uniform variables far from it decays super-exponentially, while configuration modifications usually come at an exponential cost. We can then hope to modify the configuration in the neighbourhood of the decreasing cluster of a $U$-pivotal edge in such a way that we produce a $p$-pivotal edge. Unfortunately, there exist clusters whose probability of being decreasing is only exponentially small in their volume (e.g.\ the fractal ternary square space-filling tree).

In order to remedy this, we need to only focus on a special path through a $U$-pivotal edge $e$ that we call its switching path $\cP_e$ (see Lemma~\ref{lem:path}). While influence may spread outside of the switching path, the effect of changing the value of $U_e$ on the configuration of open/closed edges is to switch the state of the edges of $\cP_e$ and nothing else. This effect is dictated by the degree constraint and resembles the role of alternating paths in dimer configurations (perfect matchings). The advantage is that, along paths, the probability of being decreasing is super-exponential, as opposed to what is the case along clusters. This allows us to pay for a modification of the uniform variables on all edges on or next to the path.

This strategy is sufficient to treat the case in which, roughly speaking, the $U$-pivotal edge is in the head of the path witnessing the modified one-arm event. However, if $e$ is in the tail, we further need to modify the configuration around the decreasing path from $e$ to the origin, in addition to the switching path. The actual case distinction is a bit more intricate, but we refer the reader to the proof of Proposition~\ref{prop:pivotality:transfer} for the details.

\section{Proof}
\subsection{The one-arm event}
\label{sec:one-arm}

For a path $P=x_1,\dots,x_m$, we denote by $\cM(P)$ the event that $p\ge U_{x_1x_2}>U_{x_2x_3}>\dots>U_{x_{m-1}x_m}$, and say that $P$ is \emph{decreasing} whenever $\cM(P)$ occurs.
\begin{definition}[Modified one-arm]
\label{def:Enp}Let $n\ge 1$ be an integer and $p\in[0,1]$. The \emph{modified one-arm event} $E_n(p)\subset\Omega$ is defined by $U=(U_e)_{e\in\cE}\in E_n(p)$, if there exists a path of vertices $(x_0,\dots,x_k)$ with $\|x_0\|_1=n>\|x_i\|_1$ for $i\in\{1,\dots,k\}$, $x_k=0$, and an integer $l\in\{0,\dots,k\}$ such that the following properties hold.
\begin{itemize}
    \item For all $i\in\{1,\dots,l\}$, we have $\omega_p(U,x_{i-1}x_i)=1$.
    \item The path $x_l,\dots,x_k$ is decreasing. 
\end{itemize}
We refer to $x_0,\dots,x_l$ as the \emph{head} and to $x_l,\dots,x_k$ as the \emph{tail} of the path.
\end{definition}

In words, the sphere of radius $n$ is connected to a vertex inside the box, from which a decreasing path reaches the origin. As we will see, this modified version of the standard one-arm event (corresponding to $l=k$ above) is motivated by the fact that, when we run the randomized algorithm to explore its occurrence, it is the modified version that occurs at edges being revealed.

 Notice that $E_n(p)$ is decreasing in $n$ and increasing in $p$ (since $p\mapsto\omega_p(U,e)$ is increasing). Moreover, $E_n(p)$ is measurable with respect to $(U_e\1_{U_e\le p})_{e\in\cE}$.

\subsection{Transforming pivotality}

\begin{definition}[Pivotal]
\label{def:pivotal}
Fix $U\in\Omega$, $p\in[0,1]$, $e\in\cE$, and an event $A\subset \Omega$. We say that $e$ is \emph{$U$-pivotal} for $A$ if there exists $U'\in\Omega$ with $U_f=U'_f$ for all $f\in\cE\setminus\{e\}$ such that $\1_A(U)\neq\1_A(U')$. We say that $e$ is \emph{$p$-pivotal} for $A$ if $U_e>p$ and, if we set $U_f=U'_f$ for all $f\in\cE\setminus\{e\}$ and $U'_e=p$, then $U'\in\Omega$ and $\1_A(U)\neq \1_A(U')$. We will only be interested in the event $A=E_n(p)$ from Definition~\ref{def:Enp}, so we do not specify it below.
\end{definition}
Write $\tau_n(p)=\bbP(E_n(p))$. The notion of $p$-pivotality is motivated by the following Russo formula (we refer the reader to Proposition 3 in \cite{Arcanjo26} for a related result).
\begin{lemma}[Russo formula]
\label{lem:russo}
For any $p$ and $n$, it holds that $\tau_n$ is differentiable on $[0,1)$ and
\begin{equation}
\label{eq:Russo}\frac{\md\tau_n(p)}{\md p}=\frac{1}{1-p}\sum_{e\in\cE}\bbP(e\text{ is $p$-pivotal}).\end{equation}
\end{lemma}
\begin{proof}
Let $\delta>0$ and write
\begin{equation}\label{russo_1} \tau_n(p+\delta)-\tau_n(p)=\bbP\left(E_n(p+\delta)\setminus E_n(p)\right),
\end{equation}
recalling that $p\mapsto E_n(p)$ is increasing. Let $\cQ_{p,\delta}=\{uv\in\cE:\|u\|_1<n,U_{uv}\in(p,p+\delta]\}$. Observe that, by Definition~\ref{def:Enp}, for $U\in E_n(p+\delta)\setminus E_n(p)$, we have $\cQ_{p,\delta}\neq\varnothing$, since $p\mapsto\omega_p(U,e)$ may only change at time $U_e$ and similarly for $p\mapsto \1_{p\ge U_e\ge U_{e'}}$, where $e,e'\in\cE$. Note that $\bbP(|\cQ_{p,\delta}|\geq 2)=o(\delta)$. Moreover, if $U\in E_n(p+\delta)\setminus E_n(p)$, $|\cQ_{p,\delta}|=1$ and $U_e\neq p$ for all $e\in\cE$, then the unique edge must be $p$-pivotal. Hence, \eqref{russo_1} becomes
\begin{align*}\label{russo_2}
\tau_n(p+\delta)-\tau_n(p)&{}=\bbP(E_n(p+\delta)\setminus E_n(p),|Q_{p,\delta}|=1)+o(\delta)=\sum_{e\in\E}\bbP(E_n(p+\delta)\setminus E_n(p),\cQ_{p,\delta}=\{e\})+o(\delta)\\
&{}=\sum_{e\in\E}\bbP\left(\mbox{$e$ is $p$-pivotal},U_e\in(p,p+\delta]\right)+o(\delta)=\frac{\delta}{1-p}\sum_{e\in\cE}\bbP(e\text{ is $p$-pivotal})+o(\delta). 
\end{align*}
Dividing both sides by $\delta$ and taking $\delta\to0$ yields \eqref{eq:Russo} for the right derivative. 

We may proceed similarly to the left, expressing the derivative in terms of pivotal edges at $p-$ instead of $p$. To conclude, it remains to prove that the function $p\mapsto\sum_{e\in\cE}\bbP(e\text{ is $p$-pivotal})$ is left-continuous. As noted above, the sum is over a finite set of edges, so it suffices to prove that each summand is continuous. To see this, let $\cF$ be the a.s.\ finite set of edges $uv\in\cE$ such that there is a decreasing path for $p=1$ starting at a vertex $x$ with $\|x\|_1\le n$ and ending at $u\in\bbZ^d$. Then, for any edge $e$, we have
\[\bbP(e\text{ is $p$-pivotal})=\sum_F\bbP(\cF=F,e\text{ is $p$-pivotal}).\]
The event $\{e\text{ is $p$-pivotal},\cF=F\}$ depends only on $(\1_{U_f\le p})_{f\in F}$ and on the order of $(U_f)_{f\in F}$, of which there are finitely many and whose probability of occurrence is a polynomial in $p$. Therefore, the summands above are continuous and converge uniformly, as $\sum\bbP(\cF=F)=1$, so the series is continuous as desired.
\end{proof}

\begin{lemma}[Switching path]
\label{lem:path}
    Fix $U\in\Omega\cap(0,1)^\cE$, $p\in(0,1)$, and $e\in\cE$. We define $U^{e+}$ by $U^{e+}_f=U_f$ for all $f\in\cE\setminus\{e\}$ and $U_e^{e+}=0$. We similarly define $U^{e-}$ with $U_e^{e-}=1$ instead. Then $\omega_p(U)\in\{\omega_p(U^{e+}),\omega_p(U^{e-})\}$. Moreover, the set of $f\in\cE$ such that $\omega_p(U^{e+},f)\neq\omega_p(U^{e-},f)$ is a path containing $e$ with possibly coinciding endpoints, but no other self-intersections. We refer to this path as the \emph{switching path} of $e$ and denote it by $\cP_e$. If $\cP_e=x_1,\dots,x_m$ with $e=x_ix_{i+1}$, then the paths $x_1,\dots,x_i$ and $x_m,\dots,x_{i+1}$ are decreasing.
\end{lemma}
\begin{proof}
Since in $\Omega$ there are only finite decreasing paths, it is clear that $\cE_e=\{f\in\cE:\omega_p(U^{e+},f)\neq\omega_p(U,f)\}$ is finite. First, assume that $e\not\in\cE_e$, that is, $\omega_p(U,e)=1$. Then, at any time $p'\le U_e$, the degree of both endpoints of $e$ in $\omega_{p'}(U)$ is at most $\kappa$. But then, by induction on the number of attempted updates at edges in a monotone path containing an endpoint of $e$, we have $\omega_{p'}(U,f)=\omega_{p'}(U^{e+},f)$ for all $f\neq e$ and $p'\le U_e$. Thus, $\omega_{U_e}(U)=\omega_{U_e}(U^{e+})$ and the Markovian construction of the process guarantees that $\omega_{p'}(U)=\omega_{p'}(U^{e+})$ for all $p'\ge U_e$. Applying this to $p'=p$ concludes this case.

Now assume that $\omega_p(U,e)=0$, so $e\in\cE_e$. Since $U_f\1_{U_f<U_e}=U^{e-}_f\1_{U_f^{e^-}<U_e}$ for all $f\in\cE$, this gives $\omega_{U_e}(U)=\omega_{U_e}(U^{e-})$. By the Markov property, we have $\omega_p(U)=\omega_p(U^{e-})$.

For the rest of the proof, we may assume that $U=U^{e-}$. Notice that, since each edge attempts to open once, we have $\omega_{p'}(U,f)=\omega_{p'}(U^{e+},f)$ for all $f\in\cE\setminus\cE_e$ and $p'\le p$. Similarly, for $f\in\cE_e$ we have $\omega_{p'}(U,f)=\omega_{p'}(U^{e+},f)$ if and only if $p'<U^{e+}_f$. 

Set $\cE_e=\{e_0,\dots,e_k\}$ with $e_0=e\in\cE_e$  and $U_{e_{i+1}}>U_{e_i}$ for $i\in\{1,\dots,k-1\}$. Also, let $p_i=U^{e+}_{e_i}$ for $i\in\{0,\dots,k\}$. We prove by induction on $i\in\{0,\dots,k\}$ that the edges $\cE_i=\{e_0,\dots,e_i\}$ form a path $\cP_i$ as in the statement, and moreover satisfy:
    \begin{itemize}
        \item consecutive edges in $\cP_i$ have different values of $\omega_p(U^{e+},f)$,
        \item all internal vertices\footnote{We do not view the endpoint of a path with coinciding endpoints as an internal vertex.} of $\cP_i$ have degree $\kappa$ in $\omega_p(U^{e+},f)$.
    \end{itemize}
    
    The base case is immediate. Assume the statement holds for some $i<k$. Then, for $p'\in[p_i,p_{i+1})$, the only vertices with different degrees in $\omega_{p'}(U)$ and $\omega_{p'}(U^{e+})$ are the endpoints of $\cP_i$. If the path has coinciding endpoints, then all vertices have the same degree ($\bbZ^d$ has no odd cycles), so $i=k$, which is a contradiction. Clearly, $e_{i+1}$ has to be incident with at least one endpoint of $\cP_i$, which forms a longer path respecting monotonicity. The fact that $e_{i+1}$ cannot create a self-intersection other than completing a cycle follows from the fact that internal vertices of $\cP_i$ already have degree $\kappa$. Moreover, since $\omega_{p_{i+1}}(U,e_{i+1})\neq\omega_{p_{i+1}}(U^{e+},e_{i+1})$, at least one endpoint $v$ of $\cP_i$ contained in $e_{i+1}$ has degree $\kappa$ in one of $\omega_{p_{i+1}-}(U),\omega_{p_{i+1}-}(U^{e+})$. For concreteness, let $\deg v=\kappa$ in $\omega_{p_{i+1}-}(U)$ (the other case is analogous). Then $\omega_{p_{i+1}}(U,e_{i+1})=0$, so $\omega_{p_{i+1}}(U^{ e+},e_{i+1})=1$, $\omega_{p_i}(U,e_j)=1$ and $\omega_{p_i}(U^{e+},e_j)=0$, where $e_j$ is the edge of $\cP_i$ containing $v$. Thus, $\deg v=(\kappa-1)+1$ in $\omega_{p_{i+1}}(U^{e+})$ and $\cP_{i+1}$ is still alternating.
\end{proof}

\begin{proposition}[Pivotality transfer]
\label{prop:pivotality:transfer}
    There exists $C=C(p,d)>0$, uniformly bounded over compacts of $(0,1)\times\bbN$, such that
    \[\sum_{e\in\cE}\bbP(e\text{ is $U$-pivotal})\le C\sum_{e\in\cE}\bbP(e\text{ is $p$-pivotal}).\]
\end{proposition}
\begin{proof}
Consider a $U$-pivotal edge $e$ for the event $E_n(p)$. Let $U^+\in E_n(p)$ with $U^+_f=U_f$ for all $f\in\cE\setminus\{e\}$. Similarly, define $U^-\in\Omega\setminus E_n(p)$ with $U_f^-=U_f$ for all $f\in\cE\setminus\{e\}$. We assume $U^+$ and $U^-$ (and further functions of $U$ below) to be selected in some measurable way as a function of $U$. Let $\omega^+_f=\omega_p(U^+,f)$ for all $f\in\cE$. For clarity, we denote by $\cV_e$ and $\cE_e$ the vertex and edge sets of the switching path $\cP_e$ (recall Lemma~\ref{lem:path}), respectively. We distinguish two cases.
\medskip

\noindent\textbf{Case A} (Head)\textbf{.} Assume that there is no decreasing path from a vertex in $\cV_e$ to the origin in $U^+\in E_n(p)$. We seek to modify the configuration in the neighbourhood of $\cV_e$ in order to make a $p$-pivotal edge appear. We say that a configuration $U'\in\Omega$ is \emph{good}, if 
\begin{itemize}
    \item $U'_f>p$ for all $f\in\cE$ intersecting $\cV_e$ and such that $\omega^+_f=0$ and
    \item $U'_f=U_f$ for all remaining $f\in\cE\setminus\cE_e$.
\end{itemize}
\begin{claim}
\label{claim:unchanged:config}
Any good configuration $U'$ satisfies $\omega_p(U',f)=\omega^+_f$ for all $f\in\cE\setminus\cE_e$, and $\omega_p(U',f)=\1_{U'_f\le p}$ for $f\in\cE_e$ (also see Lemma 1 in \cite{Arcanjo26}).
\end{claim}
\begin{proof}
Using Lemma~\ref{lem:path} successively for each $f\in\cE$ intersecting $\cV_e$ and such that $\omega^+_{f}=0$, we obtain that $\omega^+_{\cdot}=\omega_{p}(U',\cdot)$ for any good configuration $U'$ with $U'_{f}=U^+_f$ for all $f\in\cE_e$ with $\omega^+_f=1$. 
Yet, by construction, for any $v\in\cV_e$ and any good configuration $U''$, the number of edges $uv$ such that $U''_{uv}\le p$ is at most the degree of $v$ in $\omega^+$, which is at most $\kappa$. Therefore, in good configurations, the vertices in $\cV_e$ are unconstrained. In particular, changing $U'_f$ for $f\in\cE_e$ with $\omega^+_f=1$ does not change $\omega_p(U',g)$ for $g\neq f$.
\end{proof}

\begin{claim}
\label{claim:does:not:occur}
    No good configuration $U'$ with $U'_f>p$ for all $f\in\cE_e$ belongs to $E_n(p)$.
\end{claim}
\begin{proof}
    By Claim~\ref{claim:unchanged:config}, such configurations satisfy $\omega_p(U',f)=\omega^+_f\1_{f\not\in\cE_e}=\omega_p(U^-,f)\1_{f\not\in\cE_e}$, using Lemma~\ref{lem:path} in the second equality. Since a path witnessing $U'\in E_n(p)$ cannot use edges $f\in\cE$ with $U'_f>p$ by Definition~\ref{def:Enp}, such a path would also witness $U^-\in E_n(p)$, which contradicts the definition of $U^-$.
\end{proof}

\begin{claim}
\label{claim:occurs}
    Any good configuration $U'$ with $U'_f\le p$ for all $f\in\cE_e$ with $\omega^+_f=1$ belongs to $E_n(p)$.
\end{claim}
\begin{proof}
    By Claim~\ref{claim:unchanged:config}, we have that $\omega^+_f=\omega_p(U',f)$ for any $f\in\cE$ and such good configuration $U'$. Fix a path witnessing $U^+\in E_n(p)$. By the assumption of Case~A, its tail does not intersect $\cV_e$. Since $U'_f=U^+_f$ for all $f\in\cE$ disjoint from $\cV_e$ 
    by construction, the same path witnesses $U'\in E_n(p)$.
\end{proof}

Let $\cE_+=\{f\in\cE_e:\omega^+_f=1\}$. Let $\cF$ be a maximal subset of $\cE_+$, satisfying that no good configuration $U'$ with $\omega_p(U',f)=0$ for $f\in\cE_+\setminus\cF$ belongs to $E_n(p)$. By Claims~\ref{claim:does:not:occur} and~\ref{claim:occurs}, $\cF$ exists and $\cF\neq\cE_+$. Moreover, by the assumption of Case~A, no witness of a good configuration in $E_n(p)$ can have a tail intersecting $\cV_e$. Therefore, recalling Claim~\ref{claim:unchanged:config}, within good configurations $U'$, the event $U'\in E_n(p)$ depends only on $\1_{U'_f\le p}$ for $f\in\cE_+$. Therefore, for any good configuration $U'$ such that $\1_{U'_f\le p}=\1_{f\in\cF}$ for all $f\in\cE_+$, we have that each $f\in\cE_+\setminus \cF$ is $p$-pivotal for $U'$ for the event $E_n(p)$ (recall Definition~\ref{def:pivotal}) and, in particular, such $f$ exist.

Given sets $F\subset P'\subset\cE$ and an edge $f\in P'\setminus F$, we define the event $\cN(f,F,P')\subset\Omega$ so that $U\in\cN(f,F,P')$, if, for all $U'\in \Omega\cap [0,p]^{F}\times(p,1]^{P'\setminus F}\times\prod_{g\in\cE\setminus P'}\{U_{g}\}$, it holds that $f$ is $p$-pivotal for $E_n(p)$ in $U'$. We just proved that if $e$ is $U$-pivotal and Case A occurs, then there exists $f\in \cE_+\setminus\cF\subset\cE_e\setminus \cF$ such that $\cN(f,\cF,\cP')$ occurs, where $\cP'=\cE_e\cup\{g\in\cE:g\cap\cV_e\neq\varnothing,\omega^+_g=0\}$.

Consider a path $P=x_1,\dots,x_m$ containing $e=x_ix_{i+1}$ with possibly coinciding endpoints, and write $|P|=m-1$ for its length. We denote by $\vec P\in\{(x_1,\dots,x_{i}),(x_m,\dots,x_{i+1})\}$ the longest subpath of $P$ not containing $e$. Thus, $|\vec P|\ge \lfloor|P|/2\rfloor$. Recall that $\cM(\vec P)$ denotes the event that $\vec P$ is decreasing.

From the above, we have
\begin{align*}\bbP(e\text{ is $U$-pivotal, Case A})&{}= \sum_{P,P',F}\bbP(\cP_e=P,\cP'=P',\cF=F,e\text{ is $U$-pivotal, Case A})\\
&{}\le\sum_P\sum_{P'}\sum_F\sum_f\bbP(\cM(\vec P),\cN(f,F,P')),\end{align*}
where the sums are over paths $P$ containing $e$, sets $P'\supset P$ of edges with at least one endpoint in $P$, edge sets $F\subsetneq P$ and edges $f\in P\setminus F$. Noticing that $\cM(\vec P)$ and $\cN(f,F,P')$ are measurable with respect to the restriction of $U$ to $\vec P\subset P\subset P'$ and to $\cE\setminus P'$ respectively, we get that these events are independent. From the definitions we clearly have $\bbP(\cM(\vec P))=p^{|\vec P|}/|\vec P|!$ and $\bbP(\cN(f,F,P'))\le \bbP(f\text{ is $p$-pivotal})/(p^{|F|}(1-p)^{|P'|-|F|})$. Hence,
\begin{align*}
\bbP(e\text{ is $U$-pivotal, Case A})&{}\le \sum_P \sum_{f\in P}\frac{2^{2d(|P|+1)}2^{2d(|P|+1)}}{\lfloor |P|/2\rfloor!(p(1-p))^{2d(|P|+1)}}\bbP(f\text{ is $p$-pivotal})\\
&{}\le \sum_P\sum_{f\in P}\frac{C_1}{|P|^{|P|/3}}\bbP(f\text{ is $p$-pivotal})\le \sum_{f\in \cE}\frac{C_2}{2^{\delta(e,f)}}\bbP(f\text{ is $p$-pivotal}),\end{align*}
for some $C_1=C_1(p,d)>0$ and $C_2=C_2(p,d)>0$ bounded uniformly on compacts of $(0,1)\times\bbN$, where $\delta(e,f)$ denotes the $\ell^1$ distance between the edges $e$ and $f$. Indeed, in the last inequality, we used that the number of paths of given length in $\bbZ^d$ is exponential, while $|P|^{-|P|/3}$ decays super-exponentially. Summing over $e$, we obtain
\begin{equation}
\label{eq:Case:A}\sum_{e\in\cE}\bbP(e\text{ is $U$-pivotal, Case A})\le \frac{C}{2}\sum_{f\in\cE}\bbP(f\text{ is $p$-pivotal}).\end{equation}

\medskip

\noindent\textbf{Case B} (Tail)\textbf{.} 
Assume there exists a decreasing path from $\cV_e$ to the origin in $U^+$ (equivalently, in $U$). We will proceed similarly to Case A, but modifying the configuration in the neighbourhood of both $\cV_e$ and one such decreasing path. Consider a path $\cP_0$ witnessing $U^+\in E_n(p)$. Recalling Definition~\ref{def:Enp} and Lemma~\ref{lem:path}, we get that it intersects $\cV_e$, since otherwise $e$ cannot be $U$-pivotal. Let $\cP=x_1,\dots,x_k,\dots,x_l,\dots,x_m$, with $x_m=0$, be a walk possibly self-intersecting at vertices, but not edges, defined as follows: follow $\cP_0$ from the beginning until it first intersects $\cV_e$ at a vertex $x_k\in\cV_e$, then follow $\cP_e$ to a vertex $x_l\in\cV_e$ and then follow a decreasing path from $x_l$ to the origin $0=x_m$ not visiting $\cV_e$ again. Indeed, a suitable vertex $x_l$ exists by the assumption of Case B and, if its decreasing path to the origin intersects $x_1,\dots,x_k$ at an edge, erasing the resulting loop would give a witness of $U^+\in E_n(p)$ not visiting $\cV_e$, which is not possible, as argued above.

Let $j=\max\{i\in\{1,\dots,k\}:x_i\text{ is in the head of }\cP_0\}$. Also define $\cV=\{x_j,\dots,x_m\}\cup\cV_e$ and set $\bar\cV=\{x\in\bbZ^d:\delta(x,\cV)\le 5\}$, $\mathring\cV=\{x\in\bar \cV:\delta(x,\cV)\le 4\}$ and $\partial\cV=\bar\cV\setminus\mathring\cV$.  Let $i=\min\{h\in\{1,\dots,j\}:x_h\in\mathring\cV\}$. We fix a path $\cP'=y_0,\dots,y_r$ from $x_i$ to $e_1=(1,0,\dots,0)$ contained in $\mathring\cV\setminus\{0\}$. Indeed, $\mathring\cV$ is connected, as it is the sum of two connected sets (a $\delta$-ball of radius $4$ and $\cV$, which is the union of the vertex sets of the intersecting path $\cP_e$ and walk $x_j,\dots,x_m$) and $\{x\in\bbZ^d:\delta(x,0)=1\}\subset\mathring\cV$ is also connected. Let $\cE'$ denote the edge set of $\cP'$. We call a configuration $U'\in\Omega$ \emph{good}, if it satisfies the following.
\begin{enumerate}
    \item\label{item:i} $U'_f=U_f$ if $f\cap \bar \cV=\varnothing$,
    \item\label{item:ii} $U'_f=U_f$ if $|f\cap \bar\cV|=1$ and $\omega^+_f=1$,
    \item\label{item:iii} $U'_f>p$ if $|f\cap\bar\cV|=1$ and $\omega^+_f=0$,
    \item\label{item:iv} $U'_f=U_f$ if $f\subset\partial\cV$ and $f=x_ax_{a+1}$ for some $a\in\{1,\dots,i-2\}$,
    \item\label{item:v} $U'_f>p$ if $f\subset \partial\cV$ and $f\not\in\{x_ax_{a+1}:a\in\{1,\dots,i-2\}\}$,
    \item\label{item:vi} $U'_f=U_f$ if $f=x_{i-1}x_i$,
    \item\label{item:vii} $U'_f\le p$ if $f\in \cE'$,
    \item\label{item:viii} $U'_f>p$ if $f\in \cE\setminus(\cE'\cup\{x_{i-1}x_i\})$, $f\cap\mathring\cV\neq\varnothing$.
\end{enumerate}

\begin{lemma}
\label{lem:good:B}
    Let $U'$ be a good configuration. For any $f\in\cE$ with $|f\cap\bar\cV|\le 1$, we have $\omega_p(U',f)=\omega^+_f$. For any $f\in\cE$ with $f\subset\bar\cV$, we have $\omega_p(U',f)=\1_{U'_f\le p}$.
\end{lemma}
\begin{proof}
The statement holds trivially for edges in \ref{item:iii}, \ref{item:v}, and \ref{item:viii}. The vertices $v\in\mathring\cV$ have at most two incident edges $uv$ with $U'_{uv}\le p$ by construction, so the edges in \ref{item:vii} (these are all edges $f\subset \mathring\cV$ with $U'_f\le p$) indeed satisfy $\omega_p(U',uv)=1$. On the other hand, vertices $v\in\partial \cV$ satisfy that all $uv\in\cE$ with $U'_{uv}\le p$ also satisfy $\omega^+_{uv}=1$, so there are at most $\kappa$ of them. Therefore, the edges $f$ in \ref{item:iv} and \ref{item:vi} also satisfy $\omega_p(U',f)=1$. 

It remains to show that $\omega_p(U^+,f)=\omega_p(U',f)$ for all $f$ in \ref{item:i} and \ref{item:ii}. To do this, denote by $f_1,\dots,f_s$ the edges in \ref{item:ii} ordered so that $a\mapsto U_{f_a}$ is increasing. Setting $U_{f_0}=0$ and $U_{f_{s+1}}=p$, we show by induction on $a\in\{0,\dots,s+1\}$ that $\omega_{p'}(U',f)=\omega_{p'}(U^+,f)$ for all $p'\le U_{f_a}$ and $f\in\cE$ such that $|f\cap\bar\cV|\le 1$. The base case is trivial. Assume the induction statement is true for some $a\in\{0,\dots,s\}$. Then, for all $p'<U_{f_{a+1}}$ and $f\in\cE$ with $|f\cap\bar\cV|=1$, it holds that $\omega_{p'}(U',f)=\omega_{U_{f_a}}(U',f)=\omega_{U_{f_a}}(U^+,f)=\omega_{p'}(U^+,f)$. Since this edge set separates the edges in \ref{item:i} from the remaining edges (recall that $\cV_e\subset\mathring\cV$ and that, by Lemma~\ref{lem:path}, $\omega_{p'}(U^+,f)=\omega_{p'}(U,f)$ for all $p'\le p$ and $f\in\cE\setminus\cE_e$), we get that $\omega_{p'}(U',f)=\omega_{p'}(U^+,f)$ for all $p'<U_{f_{a+1}}$ and $f\in\cE$ with $|f\cap\bar\cV|\le 1$. If $a=s$, we are done. Otherwise, it remains to prove that $\omega_{U_{f_{a+1}}}(U',f_{a+1})=\omega_{U_{f_{a+1}}}(U^+,f_{a+1})=1$. This holds since the vertices in $\bar\cV$ are unconstrained in $U'$ as noted above. This completes the proof of the induction and the lemma.
\end{proof}

\begin{corollary}
\label{cor:p-pivotal}
    The edge $0e_1$ is $p$-pivotal for any good configuration.
\end{corollary}
\begin{proof}
    Fix a good $U'$. Notice that, by construction, all edges containing $0$ are in \ref{item:viii}, so $U'\not\in E_n(p)$ by Definition~\ref{def:Enp}. Moreover, there is exactly one edge of the form $xe_1$ with $U_{xe_1}\le p$, so it suffices to show that, in $U'$, each edge $f$ in the path $x_1,\dots,x_i,y_1,\dots,y_r$ satisfies $\omega_p(U',f)=1$. Inspecting the definition of good configurations, Lemma~\ref{lem:good:B} completes the proof. Indeed, the edges $f$ in $x_1,\dots,x_i$ are in the head of $\cP_0$ by definition of $j$ and $i$, so they satisfy $\omega^+_f=1$ and $U'_f=U_f\le p$.
\end{proof}

With Corollary~\ref{cor:p-pivotal} at hand, we conclude essentially as in Case A, but we spell out the details for the reader's convenience. Given disjoint edge sets $E_+,E_-\subset\cE$, we define the event $\cN(E_+,E_-)$ so that $U\in\cN(E_+,E_-)$, if, for all $U'\in\Omega\cap[0,p]^{E_+}\times(p,1]^{E_-}\times\prod_{f\in\cE\setminus(E_+\cup E_-)}\{U_f\}$, it holds that $0e_1$ is $p$-pivotal in $U'$. By Corollary~\ref{cor:p-pivotal}, if $e$ is $U$-pivotal and Case B occurs, then $\cN(\cE_+,\cE_-)$ occurs, where $\cE_+$ are the edges in \ref{item:vii} and $\cE_-$ are the edges in \ref{item:iii}, \ref{item:v} and \ref{item:viii}. Notice that all edges $f\in\cE$ with $f\subset\mathring \cV$ are either in \ref{item:vii} or \ref{item:viii}, so they belong to $\cE_+\cup\cE_-$.

Given $U$, we denote by $\vec \cP$ the longest (breaking ties arbitrarily) path among the following four paths:
\begin{itemize}
    \item $x_j,\dots,x_k$;
    \item $x_l,\dots,x_m$;
    \item the part of $\cP_e$ up to $e$ excluded;
    \item the part of $-\cP_e$ up to $e$ excluded, where $-\cP_e$ denotes $\cP_e$ in reverse order.
\end{itemize}
Notice that the union of the vertex sets of these paths is $\cV$, so $|\vec\cP|\ge (|\cV|-2)/4$. Moreover, $\cM(\vec\cP)$ occurs, since each of the four paths is decreasing by construction.

Thus, we get 
\begin{align*}
\bbP(e\text{ is $U$-pivotal, Case B})&{}=\sum_{V}\sum_{E_+,E_-,\vec P}\bbP(\cV=V,\cE_+=E_+,\cE_-=E_-,\vec\cP=\vec P,e\text{ is $U$-pivotal, Case B})\\
&{}\le \sum_V\sum_{E_+,E_-,\vec P}\bbP(\cM(\vec P),\cN(E_+,E_-)),\end{align*}
where the sums are over 
\begin{itemize}
    \item finite connected vertex sets $V\subset\bbZ^d$ containing $e$ and $0$,
    \item disjoint edge sets $E_+,E_-\subset\cE$ whose elements are contained in $\{x\in\bbZ^d:\delta(x,V)\le 6\}$ and satisfying that all $f\in\cE$ such that $f\subset\{x\in\bbZ^d:\delta(x,V)\le 4\}$ satisfy $f\in E_+\cup E_-$,
    \item paths $\vec P$ with vertex set contained in $V$ and satisfying $|\vec P|\ge (|V|-2)/4$.
\end{itemize}

Observing that the edges of $\vec P$ are contained in $E_+\cup E_-$, we get that $\cM(\vec P)$ and $\cN(E_+,E_-)$ are independent. By definition, we have $\bbP(\cM(\vec P))=p^{|\vec P|}/|\vec P|!$ and $\bbP(\cN(E_+,E_-))\le \bbP(0e_1\text{ is $p$-pivotal})/(p^{|E_+|}(1-p)^{|E_-|})$. Hence,
\begin{align*}
\bbP(e\text{ is $U$-pivotal, Case B})&{}\le \sum_V\frac{2^{C_4|V|}}{\lceil(|V|-2)/4\rceil!(p(1-p))^{C_4|V|}}\bbP(0e_1\text{ is $p$-pivotal})\\
&{}\le C_52^{-\delta(e,0)}\bbP(0e_1\text{ is $p$-pivotal})\end{align*}
for some $C_4=C_4(d)>0$ and $C_5=C_5(d,p)>0$, uniformly bounded over compacts of $(0,1)\times\bbN$, taking into account the fact that the number of connected sets of a given size containing 0 grows exponentially. Summing over $e$ yields
\[\sum_{e\in\cE}\bbP(e\text{ is $U$-pivotal, Case B})\le \frac C2\bbP(0e_1\text{ is $p$-pivotal}).\]
Combining this with \eqref{eq:Case:A}, concludes the proof of Proposition~\ref{prop:pivotality:transfer}.
\end{proof}

\subsection{The OSSS inequality}
\label{sec:OSSS}

We briefly recall the OSSS inequality in the context of product probability spaces. Let \( I \) be a countable index set, and consider the product space \( (X^I, \pi^{\otimes I}) \), where a typical element is denoted by \( x = (x_i)_{i \in I} \). We are interested in Boolean functions \( f : X^I \to \{0,1\} \), which may depend on infinitely many coordinates of \( x \).

An \emph{algorithm} \( \mathbf{A} \) determining \( f \) reveals the coordinates of \( x \) sequentially, with each choice depending on the values revealed thus far. The process terminates once the value of \( f(x) \) is determined, independently of the unrevealed coordinates. For a formal definition of randomized algorithms, we refer to~\cite{ODonnell05}. Associated to \( \mathbf{A} \) and \( f \), we define the \emph{revealment} and \emph{influence} of coordinate \( i \in I \) by
\[
\delta_i(\mathbf{A}) := \pi^{\otimes I}\left(\mathbf{A} \text{ reveals } x_i\right), \quad 
\mathrm{Inf}_i(f) := \pi^{\otimes I}\left(f(x) \neq f(x^i)\right),
\]
where \( x^i \) is obtained from \( x \) by resampling the \( i \)-th coordinate independently according to \( \pi \), leaving the others unchanged. The OSSS inequality~\cite{ODonnell05} states that for any Boolean function \( f : X^I \to \{0,1\} \) and any algorithm \( \mathbf{A} \) that determines \( f \),
\begin{equation}\label{OSSS_in}
\mathrm{Var}(f) \leq \sum_{i \in I} \delta_i(\mathbf{A}) \cdot \mathrm{Inf}_i(f).
\end{equation}

This inequality was initially formulated for finite sets $X$ and $I$. Nonetheless, the result remains valid when $X$ is a general space and $I$ is countable (see \cite{Duminil-Copin19a}*{Remark 5}). Applying \eqref{OSSS_in} to $X=[0,1]$, $\pi$ the Lebesgue measure over $X$ and $I=\cE$, so that $\pi^{\otimes I}=\bbP$, and the function $f=\1_{E_n(p)}$, and setting $\delta(\mathbf{A})=\max_{e\in\cE}\delta_e(\mathbf{A})$, we obtain
\begin{equation}
\label{eq:putting:together}
\tau_n(p)(1-\tau_n(p))\le \delta(\mathbf{A})\sum_{e\in\cE}\mathrm{Inf}_e(\1_{E_n(p)})\le C\delta(\mathbf{A})\sum_{e\in\cE}\bbP(e\text{ is $p$-pivotal})=C(1-p)\delta(\mathbf{A})\frac{\md\tau_n(p)}{\md p},
\end{equation}
using $\mathrm{Inf}_e(\1_{E_n(p)})\le \bbP(e\text{ is $U$-pivotal})$ and Proposition~\ref{prop:pivotality:transfer} in the second inequality and Lemma~\ref{lem:russo} for the last equality.

\subsection{The randomized algorithm approach}
\label{sec:DCRT}
In this section, we apply the method of \cite{Duminil-Copin19}*{Section 3} to produce a randomized algorithm with low revealment and deduce a sharpness result for the modified one-arm event $E_n(p)$. Since there is little novelty here, the presentation is rather concise and we recommend referring to \cite{Duminil-Copin19} for more details.
\begin{lemma}[Differential inequality implies sharpness \cite{Duminil-Copin19}*{Lemma~3.1}]
\label{exp_lemma} Fix $M>0$. Consider a converging sequence of differentiable functions $f_n:[p_-,p_+]\to[0, M]$ satisfying \[f'_n\geq \frac{n}{\Sigma_n}f_n,\]
for all $n\geq 1$, where $\Sigma_n=\sum_{k=0}^{n-1}f_k$. Then, there exists $p_*\in[p_-,p_+]$ such that
\begin{enumerate}
    \item for any $p<p_*$, there exists $c_{p}>0$ such that for any $n$ large enough, $f_n(p)    \leq \exp(-c_{p}n)$;
    \item for any $p>p_*$, $f=\displaystyle\lim_{n\to \infty}f_n$ satisfies $f(p)\ge p-p_*$.
\end{enumerate}
\end{lemma}

Recall that $\tau_n(p)=\bbP(E_n(p))$, with the convention $\tau_0(p)=1$, and write $\Sigma_n=\sum_{r=0}^{n-1}\tau_r(p)$. 

\begin{lemma}\label{inf_bound}
For any $n\geq 1$, there exists a randomized algorithm $\mathbf A$ determining $\1_{E_n(p)}$ with
$\delta(\mathbf A)\le 10\Sigma_n/n$.\end{lemma}

\begin{proof}
We call a path of vertices $(x_0,\dots,x_k)$ in $\bbZ^d$ \emph{nice} if there exists $l\in\{0,\dots,k\}$ such that, for all $i\in\{1,\dots,l\}$, we have $\omega_p(U,x_{i-1}x_i)=1$ and the path $x_l,\dots,x_k$ is decreasing. We refer to $x_l$ as the \emph{middle} of the nice path. For fixed $r\in\{1,\dots,n\}$, the algorithm $\bA_r$ is defined as follows. Initialize $R_0=\{x\in\bbZ^d:\|x\|_1=r\}=:\partial_r$, $S_0=\varnothing$ and $T_0=\varnothing$. Assume $R_m\subset \bbZ^d$ and $S_m,T_m\subset \E$ are given, and proceed according to the following cases, choosing edges arbitrarily if multiple edges satisfy the conditions.

\begin{enumerate}
\item\label{algo:i} If there exist $xy\in T_m$ with $x\in R_m$ and $y\not\in R_m$, we set $S_{m+1}=S_m$, $T_{m+1}=T_m$ and $R_{m+1}=R_m\cup \{y\}$.
\item\label{algo:ii} Otherwise, if there exists $xy\in S_m\setminus T_m$ such that $\|x\|_1<n$ and the restriction of $U$ to $S_m$ is sufficient to establish that $\omega_p(U,xy)=1$, we let $R_{m+1}=R_m$, $S_{m+1}=S_m$ and $T_{m+1}=T_m\cup\{xy\}$.
\item\label{algo:iii} Otherwise, if there exists $xy\in\cE$ and a decreasing path $x_l,\dots,x_k=x$ (possibly consisting of a single vertex) with edges in $S_m$ and $x_l\in R_m$, we set $R_{m+1}=R_m$, $S_{m+1}=S_m\cup\{xy\}$ and $T_{m+1}=T_m$.
\item\label{algo:iv} Otherwise, if $R_m\not\supset \partial_n$ and there is a decreasing path with edges in $S_m$ from $\partial_r$ to $0$, set $R_{m+1}=R_m\cup\partial_n$, $S_{m+1}=S_m$, $T_{m+1}=T_m$.
\end{enumerate}

Let us make a few observations about this algorithm. Firstly, whenever \ref{algo:i} does not occur, then $R_m$ is the set of vertices that can be reached from $R_0'$ via edges in $T_m$, where 
\[R'_0=\begin{cases}\partial_r&R_m\not\supset\partial_n,\\
\partial_r\cup\partial_n&R_m\supset\partial_n.
\end{cases}\]
Therefore, whenever we apply \ref{algo:iii}, there is a path of edges in $T_m$ from $R_0'$ to $x_l$, which are therefore known to be $p$-open based on the restriction of $U$ to the set $S_m$ of explored edges. In particular, whenever an edge is explored, we have discovered a nice path from $R_0'$ to an endpoint of this edge. 

Since $U\in\Omega$, the algorithm clearly terminates. Assume the algorithm has terminated. Then, for any edge $xy\in \cE\setminus T_m$ and with $\|x\|_1<n$ and $\{x,y\}\cap R_m\neq\varnothing$, we have $xy\in S_m\setminus T_m$ and $\omega_p(U,xy)=0$. Indeed, edges forming a decreasing path starting at the above edge $xy$ are certainly in $S_m$ (by \ref{algo:iii}) and are sufficient to determine $\omega_p(U,xy)$. Therefore, $R_m$ is exactly the union of the connected components of the vertices in $R_0'$ in the graph with edge set $\{xy\in\cE:\|x\|_1<n,\omega_p(U,xy)=1\}$. Consequently, all nice paths starting at $R_0'$ are contained in $S_m$, and the restriction of $U$ to $S_m$ suffices to determine that they are nice. 

Finally, let us see that the algorithm does determine the value of $\1_{E_n(p)}(U)$. If $R_m\not\supset\partial_n$, then there is no decreasing path from $\partial_r$ to $0$, so a path witnessing $E_n(p)$ would need to have its head $x_l$ satisfying $\|x_l\|<r$, so we can decompose it into a $p$-open path from $\partial_n$ to $\partial_r$ and a nice path from $\partial_r$ to $0$. Therefore, the restriction of $U$ to $S_m$ is sufficient to determine whether $E_n(p)$ occurs. If, on the contrary, $R_m\supset \partial_n$, then in particular we know if there is a nice path from $\partial_n$ to $0$, which is the event $E_n(p)$.

Let us fix an edge $xy\in\cE$ and bound $\delta_{xy}(\mathbf A_r)$. If \ref{algo:iv} was applied during the algorithm, then $E_r(p)$ occurs. If $E_r(p)$ does not occur, then 
\[\sigma_x E_{|\|x\|_1-r|}(p)\cup\sigma_y E_{|\|y\|_1-r|}(p)\]
occurs, where $\sigma_z$ is the shift by $z\in\bbZ^d$. By the union bound, this gives
\[\delta_{xy}(\bA_r)\le \tau_r(p)+\tau_{|\|x\|_1-r|}(p)+\tau_{|\|y\|_1-r|}(p).\]

Now apply algorithm $\bA_r$ with probability $1/n$ for each $r\in\{1,\dots,n\}$. Then, for any $e\in\cE$,
\[\delta_e(\bA)=\frac1n\sum_{r=1}^n\delta_e(\bA_r)\le \frac5n\sum_{k=0}^{n}\tau_k(p)\le \frac{10}{n}\Sigma_n.\qedhere\]
\end{proof}

\begin{proposition}[Sharpness for the modified one-arm event]
\label{prop:sharpness:tau}
    Consider the CDP on $\bbZ^d$, $d\geq 2$, with $\kappa\in\{2,\dots,2d\}$. There exists $p_*\in(0,1]$ such that
\begin{enumerate}
\item\label{item:exp} For $p<p_*$, there exists $\delta_p>0$ such that for all $n\geq 1$, $\tau_n(p)\leq e^{-\delta_pn}.$
\item\label{item:mf} 
There exists $c>0$ such that for $p>p_*$, $\tau(p)=\displaystyle\lim_{n \to \infty}\tau_n(p)>c(p-p_*)$.
\end{enumerate}
\end{proposition}
\begin{proof} Fix arbitrary $0<p_-<p_+<1$ and note that $1-\tau_n(p)\ge (1-p)^{2d}\ge (1-p_+)^{2d}$ for $p\in[p_-,p_+]$. Combining this with Lemma~\ref{inf_bound} and \eqref{eq:putting:together}, we obtain 
\[\tau_n'(p)\ge \frac{n}{C\Sigma_n}\tau_n (p),\]
for a constant $C\ge 1$ depending only on $p_-,p_+,d$. We take $p_-=1-p_+<1/2$ and apply Lemma~\ref{exp_lemma} to $f_n=C\tau_n$ to obtain $p_*(d,p_-)\in[p_-,p_+]$. Since, for $p<1/(2d-1)$,  by a Peierls union bound, $\tau_n(p)\le 2d(2d-1)^{n-1}p^n$, we have that $p_*(d,p_-)\ge 1/(2d-1)$. Let $p_*(d)=\limsup_{p_-\to 0}p_*(d,p_-)\ge 1/(2d-1)>0$. Item \ref{item:exp} is clear. If $p_*(d)=1$, \ref{item:mf} is void, so assume $p_*(d)<1$. For $p_-<1-p_*(d)$, we necessarily obtain $p_*(d,p_-)=p_*(d)$, since $\tau_n(p)\le e^{-\delta_pn}$ is incompatible with $\tau(p)>0$. The conclusion follows since $\tau_n(p)$ is nondecreasing in $p$.
\end{proof}

\begin{proof}[Proof of Theorem \ref{th:main}]
It follows from Definition~\ref{def:Enp} that
$\theta_n(p)\leq \tau_n(p)$, for all $n\geq 1$ and for all $p\in[0,1]$. Hence, for all $p<p_*$, we have $\theta_n(p)\leq \exp(-\delta_p n)$ for all $n\geq 1$ by Proposition~\ref{prop:sharpness:tau}. 

Write $E_{\infty}(p)=\lim_{n\to \infty}E_n(p)$, and assume $p>p_*$. Let $U\in E_{\infty}(p)$. Since, $U\in\Omega$, compactness gives that there exists an integer $n_0\ge 1$ such that there is no path $x_0=0,x_1,\dots,x_k\in\bbZ^d\setminus(-n_0,n_0)^d$ with the sequence $i\mapsto U_{x_i,x_{i+1}}$ strictly increasing. But then $U\in\bigcap_{n\ge n_0} E_n(p)$ gives that there exists $x$ with $\|x\|_\infty=n_0$ such that infinitely many $p$-open paths start at $x$. By compactness, this yields that $\omega_p(U)$ contains an infinite cluster, so $\theta(p)>0$. Hence, $p_c=p_*$, and the proof is complete. 
\end{proof}

\section*{Acknowledgements}
We thank Lyuben Lichev for a particularly enlightening discussion around the space-filling tree referred to in Section~\ref{sec:outline}. Part of this work was carried out during Roger Silva's visit to the Institut Camille Jordan, Université Claude Bernard Lyon 1, CNRS. He gratefully acknowledges their hospitality. Roger Silva was partially supported by FAPEMIG, grant APQ-06547-24. For the purpose of Open Access, a CC-BY public copyright licence has been applied by the authors to the present document and will be applied to all subsequent versions up to the Author Accepted Manuscript arising from this submission.

\appendix
\section{Proof of Corollary~\ref{cor:size}}
\label{app}

The proof relies on a one-step renormalisation scheme. For $N,M\in \bbN$ and $x=(x_1,\dots,x_d)\in \bbZ^d$, define $$B_{N,M}(x)=Nx+[-M,M)^d,$$ the box of side-length $2M$ centered at $Nx$. Fix $x\in\bbZ^d$ and consider the events
\[\cE_N(x):\mbox{the box $B_{N,N}(x)$ contains no $p$-open path of $\ell^\infty$-diameter 
at least $N/10$ inside $B_{N,\lfloor\frac{3N}{4}\rfloor}(x)$},\]
\[\cF_N(x):\mbox{there are no decreasing paths of $\ell^\infty$-diameter 
at least $N/10$ in $B_{N,N}(x)$}.\]

We introduce a renormalised lattice, isomorphic to $\bbZ^d$, whose vertices correspond to the boxes $$\{B_{N,N}(x), x\in\bbZ^d\}.$$ Given a configuration $U\in\Omega$ and $x\in\bbZ^d$, we say that the box $B_{N,N}(x)$ is \textit{good} in $U$ if $\cE_N\cap\cF_N$ occurs; otherwise it is \textit{bad}.

We claim that the resulting renormalised process of good and bad boxes is 1-dependent (with respect to the $\|\cdot\|_\infty$ norm). To see this, it suffices to show that the event $\cE_N(x)\cap\cF_N(x)$ is measurable with respect to $(U_e)_{e\subset B_{N,N}(x)}$. Indeed, the presence of decreasing paths is clearly measurable and, on $\cF_N(x)$, the $p$-open edges in $B_{N,\lfloor 3N/4\rfloor}(x)$ only depend on $U_e$ for $e\subset B_{N,\lceil N(3/4+1/10)\rceil}\subset B_{N,N}$.

Fixing $x\in\bbZ^d$ and $N>10$, let us estimate the probabilities of $\cE_N(x)$ and $\cF_N(x)$. By the union bound, translation invariance, and Theorem \ref{th:main}, we obtain
\begin{equation}\label{open_path}
\bbP(\cE_N^c)\leq (2 
N)^d e^{-\alpha_p N/10}.
\end{equation}
Similarly, by the union bound, we have 
\begin{equation}\label{decre_path}
\bbP(\cF_N^c)\leq \frac{(2N)^d 2d(2d-1)^{\lfloor N/10\rfloor-1}}{\lfloor N/10\rfloor!}.
\end{equation}
Combining \eqref{open_path} and \eqref{decre_path}, we conclude that 
\[\lim_{N\to\infty}\bbP(B_{N,N}(x)\mbox{ is good})=1.\] 

By the Liggett-Schonmann-Stacey theorem \cite{Liggett97}, one can choose $N$ sufficiently large such that this 1-dependent renormalised process stochastically dominates a supercritical independent site percolation model. For such a model, the size of the cluster of vertices corresponding to bad boxes (which is subcritical) is known to have an exponential tail (see Theorem 6.75 in \cite{Grimmett99}). The result then follows from the observation that the $p$-open cluster $C_0$ of the origin is contained in the box $B_{N,N}(0)$ together with the cluster of bad boxes containing the origin of the renormalised lattice.

\bibliographystyle{plain}
\bibliography{Bib}

\end{document}